\documentclass[10pt]{amsart}
\usepackage{latexsym} 
\usepackage{amsmath}
\usepackage{amssymb}
\usepackage{amsthm}
\usepackage{amsrefs}
\usepackage{color}
\usepackage{dsfont}
\usepackage{mathrsfs}
\usepackage{caption}
\usepackage{subcaption}
\usepackage{mathtools}
\usepackage{stmaryrd}
\usepackage[all]{xy}
\usepackage[mathcal]{eucal}
\usepackage{verbatim}  
\usepackage{geometry}
\theoremstyle{plain}
  \newtheorem{theorem}{Theorem}[section]
  \newtheorem{proposition}[theorem]{Proposition}
  \newtheorem{lemma}[theorem]{Lemma}
  \newtheorem{corollary}[theorem]{Corollary}
  
  \newtheorem{question}[theorem]{Question}
\theoremstyle{definition}
  
  \newtheorem{example}[theorem]{Example}
 \theoremstyle{remark}
  \newtheorem{remark}[theorem]{Remark}

\numberwithin{equation}{section}

\def\FF{{\mathbb F}}
\def\ZZ{{\mathbb Z}}
\newcommand{\F}{\mathbb{F}}

\def\CCC{{\mathcal C}}

\def\maj{{\operatorname{maj}}}
\def\inv{{\operatorname{inv}}}
\def\cdes{{\operatorname{cdes}}}
\def\wt{{\operatorname{wt}}}
\def\stat{{\operatorname{stat}}}

\def\symm{{\mathfrak{S}}}

\def\zero{{\mathbf{0}}}

\begin{document}

\title[Problem 6: Cyclic Sieving for cyclic codes]
{ Cyclic Sieving for cyclic codes}

\author{Alex Mason}
\address{Department of Mathematics, University of Washington, 
Seattle WA 98105, USA}
\email{acmason@uw.edu}

\author{Victor Reiner}
\address{School of Mathematics, University of Minnesota, 
Minneapolis, MN 55455, USA}
\email{reiner@umn.edu}

\author{Shruthi Sridhar}
\address{Department of Mathematics, Princeton University, 
Princeton NJ 08544}
\email{ssridhar@math.princeton.edu}


\thanks{This research was carried out as part of the 2017 REU program at the School of Mathematics
at University of Minnesota, Twin Cities. The authors are grateful for the support of NSF RTG
grant DMS-1148634.}


\date{\today}

\begin{abstract} 
Prompted by a question of Jim Propp, this paper examines the cyclic sieving phenomenon (CSP) in certain cyclic codes.  For example, it is shown that, among
dual Hamming codes over $\FF_q$, the generating function for codedwords according to the major index statistic (resp. the inversion statistic) gives rise to a CSP when $q=2$ or $q=3$ (resp. when $q=2$).  A byproduct is a curious characterization of the irreducible polynomials in $\FF_2[x]$
and $\FF_3[x]$ that are primitive.
\end{abstract}

\maketitle

\section{Introduction}

The Cyclic Sieving Phenomenon describes the following enumerative situation.  One has a finite set $X$ having the action of a cyclic group $C=\langle c \rangle=\{1,c,c^2,\ldots,c^{n-1}\}$ 
of order $n$, and a polynomial $X(t)$ in $\ZZ[t]$ that not only satisfies $\#X=X(1)$, but furthermore every element $c^d$ in $C$ satisfies
$$
\#\{x \in X: c^d(x)=x\} = \left[ X(t) \right]_{t=\left(e^{\frac{2\pi i}{n}}\right)^d}.
$$
In this case, one says that the triple $(X,X(t),C)$ {\it exhibits the cyclic sieving phenomenon (CSP)}; see
\cite{RStantonWhite} for background and many examples.
Frequently the polynomial $X(t)$ is a {\it generating function} 
$$
X^{\stat}(t):=\sum_{x \in X} t^{\stat(x)}
$$
for some combinatorial statistic $X \overset{\stat}{\longrightarrow} \{0,1,2,\ldots\}$.  Some of the first examples of CSPs
(e.g., Theorem~\ref{general-word-CSP} below)
arose for the cyclic $n$-fold rotation action on
certain special collections $X$ of words
$w=(w_1,\ldots,w_n)$ of length $n$ in a linearly ordered
alphabet, with $X(q)=X^{\maj}(t)$ or $X^{\inv}(t)$ being generating functions for the
{\it major index} and 
{\it inversion number} statistics, defined as follows:
\begin{equation}
\begin{aligned}
\inv(w)&:=\#\{(i,j): 1 \leq i < j \leq n: w_i > w_j\},\\
    \maj(w)&:=\sum_{i: w_i > w_{i+1}} i.
\end{aligned}
\end{equation}

This prompted Jim Propp to ask the question \cite{Propp}
of whether there are such CSPs in which $X$ is the
set of codewords $\CCC$ for a cyclic error-correcting code.  He initially observed the following instances of CSP triples $(X,X^\stat(t),C)$ where $X=\CCC$ is
a cyclic code inside $\FF_q^n$, and $C=\ZZ/n\ZZ$ acts as $n$-fold cyclic rotation of words, and   either\footnote{There is a subtlety here: one must choose a linear order arbitrarily on $\FF_q$ 
to define the statistics $\maj, \inv$.  However, it turns out that none of the results that we discuss here, including Propp's observations, will depend upon this choice.} $\stat=\maj$ or $\stat=\inv$: 
\begin{itemize}
    \item all repetition codes (trivially),  
    \item the full codes $\CCC=\FF_q^n$, 
    \item all parity check codes, and 
    \item all binary cyclic codes of length $7$.
\end{itemize}
After a quick review of cyclic codes in Section~\ref{codes-preliminaries-section},
a few simple observations about CSPs for cyclic actions on words 
in Section~\ref{CSP-preliminaries-section} will explain
all of the above CSPs, and a few more.

Section~\ref{Hamming-section} addresses the more subtle
examples of {\it dual Hamming codes}.  Among other things, it shows that either $X^\maj(t)$ or $X^\inv(t)$ give
rise to a CSP for all {\it binary} dual Hamming codes,
while $X^\maj(t)$ also works for all {\it ternary} Hamming codes.  The analysis
leads to a curious characterization 
(Theorem~\ref{primitive-polynomial-conjecture}(ii)) of which irreducible polynomials
in $\FF_2[x]$ or $\FF_3[x]$ are {\it primitive} polynomials.

\section{Preliminaries}
\label{codes-preliminaries-section}

We  briefly review here the notions of linear codes, cyclic codes, and the examples that we will consider; see, e.g., Garrett
\cite{Garrett-error-correcting-codes} or Pless \cite{Pless} for more background.
Recall an {\it $\FF_q$-linear code} of length $n$ is an $\FF_q$-linear 
subspace $\CCC \subseteq \FF_q^n$.  One calls $\CCC$ {\it cyclic} if it is also stable under the action of the cyclic group $C=\{e,c,c^2,\ldots, c^{n-1} \} \cong \ZZ/n\ZZ$ whose generator $c$ cyclically shifts codewords $w$ as follows:
$$
c(w_1,w_2,\ldots,w_n):=(w_2,w_3,\ldots,w_n,w_1).
$$
It is convenient to rephrase this using the $\FF_q$-vector space isomorphism 
\begin{equation}
\label{codeword-to-polynomial-dictionary}
\begin{array}{rcl}
\FF_q^n & \longrightarrow & \FF_q[x]/(x^n-1) \\
w=(w_1,\ldots,w_n) & \longmapsto & w_1+w_2x+w_3x^2+\cdots+w_nx^{n-1}.
\end{array}
\end{equation}
After identifying a code $\CCC \subset \FF_q^n$ with its image under the isomorphism in \eqref{codeword-to-polynomial-dictionary}, the $\FF_q$-linearity of $\CCC$ together with cyclicity means that $\CCC$ forms an {\it ideal} within the principal ideal ring $\FF_q[x]/(x^n-1)$.  Hence $\CCC$ is the set $(g(x))$ of all multiples of some {\it generating polynomial} $g(x)$.   This means that
$$
\CCC= \{ h(x)g(x) \in  \FF_q[x]/(x^n-1): \deg(h) + \deg(g) < n \}
$$
and therefore 
$
k:=\dim_{\FF_q} \CCC=n-\deg(g(x)).
$ 
The {\it dual code} $\CCC^\perp$ of a linear code $\CCC$
in $\FF_q^n$ is defined as
$$
\CCC^\perp:=\{ v \in \FF_q^n: 0=v \cdot w = \sum_{i=1}^n v_i w_i\}.
$$
One has that $\CCC$ is cyclic with generator $g(x)$ if and only if $\CCC^\perp$ is cyclic with generator $
g^\perp(x):=\frac{x^n-1}{g(x)},
$
called the {\it parity check polynomial} for the primal code $\CCC$.  This implies 
$
k=\dim_{\FF_q} \CCC=\deg(g^\perp(x)).
$

\begin{example}
The cyclic code $\CCC$ having $g^\perp(x)=1+x+x^2+\cdots+x^{n-1}$ is called the {\it parity check} code of length $n$
(particularly when $q=2$). As a vector space, it is the space of all vectors in $\FF_q^n$ with coordinate sum 0. Its dual code $\CCC^\perp$ consisting of the scalar multiples of  $g^\perp(x)=1+x+x^2+\cdots+x^{n-1}$ is the {\it repetition code}.
For example, the ternary ($q=3$) repetition code $\CCC^\perp$ and
parity check code $\CCC$ of length $n=2$, and their respective generator polynomials $g^\perp(x), g(x)$ inside $\FF_3[x]/(x^2-1)$, are 
$$
\begin{array}{rcr}
\CCC^\perp
  =\{[0,0],[1,1],[2,2]\},
  &\text{ generated by }& g^\perp(x)=1+x,\\
\CCC
 =\{[0,0],[1,2],[2,1]\},
  &\text{ generated by }& g(x)=\frac{x^2-1}{1+x}=1+2x.
\end{array}
$$
\end{example}

\begin{example}
Recall that a degree $k$ polynomial $f(x)$ in $\FF_q[x]$ is called {\it primitive} if it is not only irreducible, but also has the property that
the image $\bar{x}$ of the variable $x$ in the finite field $\FF_q[x]/(f(x))$ has the maximal possible multiplicative order, namely $n:=q^k-1$.
Equivalently, $f(x)$ is primitive when it is irreducible 
but divides none of the polynomials $x^d-1$ for proper divisors $d$ of $n$.

A cyclic code $\CCC$ generated by a primitive polynomial $g(x)$ in $\FF_q[x]$ of degree $k$ is called a {\it Hamming code} of
length $n=q^k-1$ and dimension $n-k$.  Its dual $\CCC^\perp$ generated by $\frac{x^n-1}{g(x)}$ 
is a {\it dual Hamming code} of length $n$ and dimension $k$.
See Example~\ref{dual-Hamming-example} below for some
examples with $q=3$ ({\it ternary} codes) with $k=2$
and length $n=3^2-1=8$.
\end{example}

Propp suggested looking for CSPs $(X,X(t),C)$ with $X$ a Hamming or dual Hamming code, using $X(t)=X^\maj(t)$ or $X^\inv(t)$.
This seems not to happen often for Hamming codes, but 
Section~\ref{Hamming-section} exhibits many dual Hamming codes with such a CSP.

\section{Preliminaries on CSPs for words}
\label{CSP-preliminaries-section}

We first explain why the CSPs for full codes $\FF_q^n$ and parity check codes are special cases of a general CSP for words, Theorem~\ref{general-word-CSP} below, which follows from a result in \cite{RStantonWhite}; see \cite[Prop. 17]{BergetEuR}.  

Let $A$ be a linearly ordered alphabet,
and consider the set $A^n$ of all words $w=(w_1,w_2,\ldots,w_{n-1},w_n)$ of length
$n$ in the alphabet $A$.  As before,
let the cyclic group $C=\{1,c,c^2,\ldots,c^{n-1}\}$ act on $A^n$ via $n$-fold rotation,
so $c(w)=(w_n,w_1,w_2,\ldots,w_{n-1})$.

\begin{theorem}
\label{general-word-CSP}
Let $X \subseteq A^n$ by any collection of words which is
stable under the symmetric group $\symm_n$ acting on the $n$ positions, 
that is, if $w=(w_1,\ldots,w_n)$ in $X$ then $(w_{\sigma(1)},\ldots,w_{\sigma(n)})$ is also in $X$ for
every $\sigma$ in $\symm_n$.  Then $(X,X^\stat(t),C)$ exhibits the CSP, where either $\stat=\maj$ or 
$\stat=\inv$.
\end{theorem}

\noindent
Note that Theorem~\ref{general-word-CSP} 
explains Propp's observation of CSP triples
involving either the full codes $\CCC=\FF_q^n$ or the 
parity check codes 
$
\CCC=\{w \in \FF_q^n: \sum_{i=1}^n w_i = 0 \},
$
since both are $\symm_n$-stable inside $\FF_q^n$.

\medskip

The next proposition analyzes how $\maj(w)$ changes\footnote{A much 
more sophisticated analysis may be found in Ahlbach and Swanson \cite{AhlbachSwanson}.}
when applying the cyclic shift $c$ to the word $w$,
and similarly for $\inv(w)$ if the alphabet $A$ is binary.  In the latter case, we assume $A=\{0,1\}$ has linear order $0<1$, and will refer to the
{\it Hamming weight}  $\wt(w)$, as the number of ones in $w$.  We also use another statistic on words $w$ in $A^n$, 
the number of {\it cyclic descents}
$$
\cdes(w):=\#\{i: 1 \leq i \leq n \text{ and } w_i > w_{i+1}\}
$$
where we decree $w_{n+1}:=w_1$ to understand the inequality $w_i > w_{i+1}$ when $i=n$.  Lastly, define a {\it $t$-analogue} of the number $n$
by this geometric series:
$
[n]_t:=1+t+t^2+\cdots+t^{n-1}=\frac{t^n-1}{t-1}.
$
The following proposition is then straightforward to check from the definitions.

\begin{proposition}
\label{maj-inv-change-under-shift-prop}
Let $A$ be any linearly ordered alphabet, and $w$ a word in $A^n$.
\begin{enumerate}
    \item[(i)] The statistic $\cdes(w)$ is constant among all words within the $C$-orbit of $w$, and
    $$
    \begin{aligned}
    \maj(c(w))&=
    \begin{cases} 
    \maj(w)+\cdes(w) & \text{ if }w_n \leq w_1, \\
    \maj(w)+\cdes(w)-n & \text{ if }w_n > w_1 
    \end{cases}\\
    &\equiv \maj(w)+\cdes(w)\, \bmod{n}.
    \end{aligned}
    $$
    \item[(ii)] In the binary case $A=\{0,1\}$, one has
    $$
    \begin{aligned}
    \inv(c(w))&=
    \begin{cases}
    \inv(w)+\wt(w) & \text{ if }w_n=1,\\
    \inv(w)+\wt(w)-n & \text{ if }w_n=0 
    \end{cases}\\
    &\equiv \inv(w)+\wt(w)\, \bmod{n}.
    \end{aligned}
    $$
    
\end{enumerate}
\end{proposition}

The congruences modulo $n$ in 
Proposition~\ref{maj-inv-change-under-shift-prop}
immediately imply the following.

\begin{proposition}
\label{maj-inv-on-free-orbits-prop}
When $w$ in $A^n$ has free $C$-orbit, meaning that
$\{w,c(w),c^2(w),\ldots,c^{n-1}(w)\}$ are all distinct, then
one has the following congruence in $\ZZ[t]/(t^n-1)$:
$$
X^\maj(t) \equiv t^{\maj(w)} \cdot [n]_{t^{\cdes(w)}} \, \bmod{t^n-1}.
$$
In the binary case, one has
$$
X^\inv(t) \equiv t^{\inv(w)} \cdot [n]_{t^{\wt(w)}} \, \bmod{t^n-1}.
$$
\end{proposition}

The next corollary then explains Propp's observation about
CSPs for binary cyclic codes $X=\CCC$ of length $n=7$ using either $X^\maj(t)$ or $X^\inv(t)$. The key point is that $7$ is prime.  We will also frequently use the fact that the following three conditions are equivalent for positive integers $k, n$:
\begin{itemize}
    \item $\gcd(k,n)=1$.
    \item $t^\ell [n]_{t^k} \equiv [n]_t \, \bmod{t^n-1}$ for all nonnegative integers $\ell$.
    \item $t^\ell [n]_{t^k}$ vanishes upon evaluating $t$ at any $n$th root-of-unity that is not $1$.
\end{itemize}

\begin{corollary}
\label{n-prime-corollary}
When $n$ is prime, every $C$-stable subset $X \subset A^n$
gives rise to a CSP triple $(X,X^\maj(t),C)$.  If furthermore,
$A=\{0,1\}$, then one also has the CSP triple $(X,X^\inv(t),C)$.
\end{corollary}
\begin{proof}
Since $n$ is prime, $X$ must consist of a certain number $m$ of singleton $C$-orbits that each consist of a single constant word, along with
a list of free $C$-orbits labeled $X_1,\ldots,X_r$, say with $C$-orbit 
representatives labeled $w^{(1)}, \ldots, w^{(r)}$.
Then for either statistic $\stat=\maj$ or $\stat=\inv$, the polynomial $X^\stat(t)$ satisfies $\#X=X(1)$ by definition.  For $c^d \neq 1$ in $C$, one has 
$$
\#\{x \in X: c^d(x)=x\}=m,
$$
and so it remains to check
that when $c^d \neq 1$ one has
$\left[ X^\stat(t) \right]_{t=\zeta_n^d}=m$ if $\zeta_n:=e^{\frac{2\pi i}{n}}$.
Since 
$$
X^\stat(t)= m + \sum_{i=1}^r X^{\stat}_i(t),
$$
it suffices to show 
for each $i=1,2,\ldots,r$ that $[X_i^{\stat}(t)]_{t=\zeta_n^d}=0$.

Using Proposition~\ref{maj-inv-on-free-orbits-prop}, 
this follows for $\stat=\maj$ since
$$
\begin{aligned}
X_i^{\maj}(t) 
&\equiv t^{\maj(w^{(i)})} \cdot [n]_{t^{\cdes(w^{(i)}}} \, \bmod{t^n-1} \\
&\equiv  [n]_t \, \bmod{t^n-1}
\end{aligned}
$$
where the last congruence above comes from $n$ being prime
and $1 \leq \cdes(w) \leq n-1$ for any non-constant word $w^{(i)}$, so that $\gcd(n,\cdes(w))=1$.  

If $A=\{0,1\}$, it similarly follows using Proposition~\ref{maj-inv-on-free-orbits-prop} for $\stat=\inv$,
as
$$
\begin{aligned}
X_i^{\inv}(t) 
&\equiv t^{\inv(w^{(i)})} \cdot [n]_{t^{\wt(w)}} \, \bmod{t^n-1} \\
&\equiv  [n]_t \, \bmod{t^n-1}.
\end{aligned}.
$$
The last congruence holds as $w^{(i)}$ non-constant
gives $1 \leq \wt(w^{(i)}) \leq n-1$ so $\gcd(n,\wt(w^{(i)}))=1$.
\end{proof}

\begin{remark}
Note that
Proposition~\ref{maj-inv-on-free-orbits-prop} dashes
any false hopes one might have that
for {\it binary} words $w$ in $\{0,1\}^n$, the
distributions
$
\sum_{j=0} t^{\maj(c^j(w))}
$
and 
$
\sum_{j=0} t^{\inv(c^j(w))}
$
are congruent modulo $t^n-1$.
This can fail for non-prime $n$ even when $w$ has a free
$C$-orbit. For example, $w = (1,0,1,1,0,0)$ has 
$$
\sum_{j=0} t^{\maj(c^j(w))} \equiv t^5[6]_{t^2}
\not\equiv t^7[6]_{t^3}
\equiv \sum_{j=0} t^{\inv(c^j(w))}
\, \bmod{t^6-1}.
$$
\end{remark}

Our discussion of dual Hamming codes will use another consequence of 
Proposition~\ref{maj-inv-on-free-orbits-prop}.

\begin{corollary}
\label{single-free-orbit-corollary}
Suppose that a $C$-stable subset $X \subseteq A^n$ of words
has all non-constant words in $X$ lying in a single free $C$-orbit, represented by the word $w$.  
\begin{enumerate}
    \item[(i)] Then
$(X,X^\maj(t),C)$ gives rise to a CSP triple if and only 
if $\gcd(n,\cdes(w))=1$. 
\item[(ii)] In the binary case, $(X,X^\inv(t),C)$ gives rise to a CSP triple if and only 
if $\gcd(n,\wt(w))=1$. 
\end{enumerate}
\end{corollary}
\begin{proof}
As in the proof of 
Corollary~\ref{n-prime-corollary}, for either statistic $\stat=\maj$ or $\stat=\inv$, 
the CSP holds if and only if $\sum_{j=0}^{n-1} t^{\stat(c^j(w))}$
vanishes upon setting $t=\zeta_n^d$ 
for any $n^{th}$ root-of-unity $\zeta_n^d \neq 1$.  Since Proposition~\ref{maj-inv-on-free-orbits-prop} implies the above sum is congruent modulo $t^n-1$ to $t^{\maj(w)} [n]_{t^{\cdes(w)}}$ when $\stat=\maj$,
and to $t^{\inv(w)} [n]_{t^{\wt(w)}}$ in the binary case when $\stat=\inv$, the result follows.
\end{proof}

\section{Dual Hamming codes}
\label{Hamming-section}

To understand when dual Hamming codes $\CCC$ exhibit a
CSP, it will help to have many ways to characterize
them among cyclic codes.
As a precursor step, it helps to first characterize the cyclic codes
for which the cyclic action on nonzero codewords is free.

\begin{proposition}
\label{free-cyclic-code-prop}
A cyclic code $\CCC \subset \FF_q^n$ with parity check polynomial $g^\perp(x)$
will have the $C$-action on $\CCC \setminus \{\zero\}$ free if and only if 
$\gcd( g^\perp(x), x^d-1 ) = 1$ for all proper divisors $d$ of $n$.
\end{proposition}
\begin{proof}
First note that, since $c^n=1$,
whenever a codeword $w$ in $\CCC$ is fixed by some element $c^d \neq 1$ in $C$, 
without loss of generality, one may assume $d$ is a proper divisor of $n$; otherwise replace $d$ by $\gcd(d,n)$.
When this happens, the polynomial $h(x)g(x)$ representing $w$ in $\FF_q[x]/(x^n-1)$ has
$$
x^d h(x)g(x) \equiv h(x)g(x) \bmod x^n-1
$$
or equivalently
$(x^d-1)h(x)g(x)$ is divisible by $x^n-1$ in $\FF_q[x]$.
Canceling factors of $g(x)$, this is equivalent to $(x^d-1)h(x)$ being divisible by $g^\perp(x)$ in $\FF_q[x]$.
However, as discussed in Section~\ref{codes-preliminaries-section}, 
$h(x)$ can be chosen with degree strictly less than $k=\dim \CCC= \deg(g^\perp(x))$,
so the existence of such a nonzero $h(x)$ is equivalent to $g^\perp(x)$ sharing a common factor with $x^d-1$.  
\end{proof}

The next result compiles various equivalent 
characterizations of the primitive polynomials within $\FF_q[x]$, or equivalently, the dual Hamming codes.
Although many of the equivalences are well-known (see, e.g., Garrett \cite[Chap. 21]{Garrett-crypto}, \cite[Chap. 16]{Garrett-error-correcting-codes}, Klein \cite[Chap. 2]{Klein} for some), we were unable to find a source for all of them in the literature, so we have included the proofs here.
Some of the equivalences involve the {\it linear feedback shift register} (LFSR) associated to a monic polynomial 
$f(x)=a_0+a_1x+a_2x^2+\cdots+a_{k-1}x^{k-1}+x^k$,
which is the $\FF_q$-linear map
$$
\begin{array}{rcl}
\F_q^k &\overset{T_f}{\longrightarrow}& \F_q^k\\
{\bf x}=(x_0, x_1, ... x_{k-1}) &\longmapsto & 
T_f({\bf x})=(x_1, ... x_{k-1}, x_k)
\end{array}
$$
where $x_k:=-\left(a_0 x_{k-1}+a_1 x_{k-2}+\cdots+a_{k-1}x_0\right)$. 
Starting with a {\it seed vector} ${\bf s}=(s_0, s_1, ... s_{k-1})$,
since ${\bf s}$ and $T_f({\bf s})$ overlap in a consecutive subsequence
of length $k-1$, 
it is possible to create an infinite {\it pseudorandom sequence}
$
(s_0,s_1,\ldots,s_{k-1},s_k,s_{k+1},\ldots)
$
containing as its length $k$ 
consecutive subsequences all of the iterates 
$T_f^r ({\bf s})=(s_r,s_{r+1},\ldots,s_{r+k-1})$.

\begin{proposition}
\label{primitive-dual-Hamming-equivalences-prop}
Let $g^\perp(x)$ be any monic irreducible degree $k$ polynomial in $\FF_q[x]$ that divides $x^n-1$, where $n:=q^k-1$.  
Let $\CCC \subset \FF_q^n$ be the $k$-dimensional cyclic code generated by $g(x)=\frac{x^n-1}{g^\perp(x)}$.
Then the following are equivalent:
\begin{enumerate}
    \item[(i)] The $C$-action by $n$-fold cyclic shifts on
    $\CCC\setminus \{{\bf 0}\}$ inside $\FF_q^n$ is
    simply transitive.
    \item[(ii)] $\gcd(g^\perp(x),x^d-1)=1$ for all proper divisors $d$ of $n$.
    \item[(iii)] $g^\perp(x)$ is primitive, that is, $\bar{x}$ has order $n$ in $\FF_q[x]/(g^\perp(x))$, so
    $\CCC$ is dual Hamming.
    \item[(iv)] The linear feedback shift register $T_{g^\perp}: \FF_q^k \rightarrow \FF_q^k$
    has order $n$.
    \item[(v)] With seed ${\bf s}:=(0,\cdots,0,1)$ in $\FF_q^k$, the iterates $\left\{ T_{g^\perp}^r ({\bf s}) \right\}_{r=0,1,\ldots,n-1}$ exhaust $\FF_q^k \setminus \{\bf 0\}$.
    \item[(vi)] The pseudorandom sequence
    generated by $T_{g^\perp(x)}$ with seed ${\bf s}:=(0,\cdots,0,1)$ is $n$-periodic, and each period contains each vector $\FF_q^k \setminus \{\bf 0 \}$ as a consecutive subsequence exactly once.
    \item[(vii)] The codeword $w$ in $\CCC \subset \FF_q^n$ corresponding under \eqref{codeword-to-polynomial-dictionary} to $g(x)$ in $\FF_q[x]/(x^n-1)$, when repeated $n$-periodically, has each vector of $\FF_q^k \setminus \{\bf 0 \}$ as a consecutive subsequence once per period.
\end{enumerate}
\end{proposition}

\begin{example}
\label{dual-Hamming-example}
When $q=3$ and $k=2$, so $n=3^2-1=8$,
there are three degree two monic irreducibles $g^\perp(x)$ in $\FF_3[x]$, each shown
here with $g(x)=\frac{x^8-1}{g^\perp(x)}$ and
its corresponding word $w$ in $\FF_3^8$:

\begin{center}
\begin{tabular}{|c|c|c|}\hline
$g^\perp(x)$ & $g(x)=\sum_{j=1}^n w_j x^{j-1}$ & $w=(w_1,\ldots,w_8)$ \\\hline\hline
$x^2+1$ &$x^6+2x^4+x^2+2$ & $(2,0,1,0,2,0,1,0)$\\\hline
$x^2+x+2$ & $x^6+2x^5+2x^4+2x^2+x+1$ & $(1,1,2,0,2,2,1,0)$\\\hline
$x^2+2x+2$ & $x^6+x^5+2x^4+2x^2+2x+1$ & $(1,2,2,0,2,1,1,0)$\\\hline
\end{tabular}
\end{center}

\vskip.1in
\noindent
The first choice is not primitive,
while the second and third are primitive.
The non-primitive first choice $g^\perp(x)=x^2+1$
has LFSR $L_{g^\perp(x)}: (x_0,x_1) \mapsto (x_1,x_2)$ 
where $x_2=-(0\cdot x_1+1\cdot x_0)=-x_0$.  Starting with seed $(0,1)$, it
has only $4$ different iterates
$$
(0,1) \mapsto (1,0) \mapsto (0,2) \mapsto (2,0) \mapsto (0,1) \mapsto (1,0) \mapsto (0,2) \mapsto (2,0) \,\, (\mapsto (0,1) \mapsto \cdots)
$$
and this pseudorandom sequence 
$
(0,1,0,2,0,1,0,2,0,1,0,2,0,1,\ldots),
$
whose period is $4$, not $n=8$.

The primitive second choice $g^\perp(x)=x^2+x+2$ has
LFSR $L_{g^\perp(x)}: (x_0,x_1) \mapsto (x_1,x_2)$ 
where $x_2=-(1\cdot x_1+2\cdot x_0)=-x_1+x_0$.
Starting with seed $(0,1)$ it has $8$ different iterates (all of $\FF_3^2 \setminus \{\bf 0\})$
$$
(0,1) \mapsto (1,2) \mapsto (2,2) \mapsto (2,0) \mapsto (0,2) \mapsto (2,1) \mapsto (1,1) \mapsto (1,0) \,\, (\mapsto (0,1) \mapsto \cdots)
$$
and pseudorandom sequence 
$
(0,1,2,2,0,2,1,1,0,1,2,2,0,2,1,1,\ldots),
$
whose period is $n=8$.

\end{example}

\begin{proof}
\noindent
{\sf (i) $\Leftrightarrow$ (ii):}
Since both the cyclic group $C$ and $\CCC\setminus \{\bf 0\}$
have $n=q^k-1$ elements, the $C$-action on $\CCC\setminus \{\bf 0\}$ 
is simply transitive if and only if it is free.  Proposition~\ref{free-cyclic-code-prop} then implies the equivalence.
\vskip.1in
\noindent
{\sf (ii) $\Leftrightarrow$ (iii):}
Since $g^\perp(x)$ is an irreducible factor of $x^n-1$, having
$\gcd(g^\perp(x),x^d-1)$ for all proper divisors $d$ of $n$ is
the same as saying $g^\perp(x)$ does not divide $x^d-1$ for
any proper divisor $d$ of $n$.  The latter is the same
as saying $\bar{x}$ has order $n$ inside $\FF_q[x]/(g^\perp(x))$.

\vskip.1in
\noindent
{\sf (iii) $\Leftrightarrow$ (iv):}
The matrix for $T_{g^\perp(x)}$ acting in the standard basis
for $\FF_q^n$ is the transpose of the matrix for multiplication by $\bar{x}$
acting in the ordered basis $(\bar{1},\bar{x},\bar{x}^2,\ldots,\bar{x}^{k-1})$
for $\FF_q[x]/(g^\perp(x))$, that is, the usual {\it companion matrix} for $g^\perp(x)$.  Therefore they have the same multiplicative order.

\vskip.1in
\noindent
{\sf (v) $\Rightarrow$ (iv):}
Since $T_{g^\perp(x)}$ has the same multiplicative order as
multiplication by $\bar{x}$ in $\FF_q[x]/(g^\perp(x))$, and since $g^\perp(x)$
divides $x^n-1$, the latter order divides $n$.  However, if 
the iterates $\left\{ T_{g^\perp}^r (e_k) \right\}_{r=0,1,\ldots,n-1}$ exhaust $\FF_q^k-\setminus \{\bf{0}\}$, then there are $n$ of them,
so $T_{g^\perp}$ has order at least $n$, and hence exactly $n$.

\vskip.1in
\noindent
{\sf (ii) $\Rightarrow$ (v):}  Assume (v) fails, that is, the $n$ iterates $\left\{ T_{g^\perp}^r (e_k) \right\}_{r=0,1,\ldots,n-1}$ do not
exhaust the set $\FF_q^k-\setminus \{\bf{0}\}$ of cardinality $n$, so two of them are equal.  Since
$T_{g^\perp}$ is invertible, this means $T_{g^\perp}^d({\bf x})={\bf x}$
for some ${\bf x} \neq {\bf 0}$ and $1 \leq d < n$.  Thus $T_{g^\perp}$ has
an eigenvalue $\alpha$ in $\bar{\FF}_q$ which is a $d^{th}$ root-of-unity for some proper
divisor $d$ of $n$, and hence its characteristic polynomial $g^\perp(x)$ has
$\alpha$ as a root.  But this would contradict (ii): primitivity of $g^\perp(x)$ implies that any of its
roots $\alpha$ gives rise to an isomorphism
$\FF_q[x]/(g^\perp(x)) \cong \FF_q[\alpha]$ sending $\bar{x} \mapsto \alpha$, so $\alpha$ should have order $n$.

\vskip.1in
\noindent
{\sf (v) $\Leftrightarrow$ (vi):} 
By construction the $n$-periodicity of the pseudorandom sequence comes from the fact that $T_{g^\perp}$ had the same order $n$ as $\bar{x}$.  The rest of (vi) is then a restatement of (v).

\vskip.1in
\noindent
{\sf (vi) $\Leftrightarrow$ (vii):}
We claim that the word $w$ in (vii) is the reverse of the pseudorandom sequence in (vi).  This is because the equation
$$
x^n-1=g^\perp(x) g(x)=\left( x^k+\sum_{i=0}^{k-1}a_i x^i \right)
\left(\sum_{j=1}^{n} w_j x^{j-1} \right)
$$
defining $g(x)$ via $g^\perp(x)$ makes the coefficient of $x^m$ vanish on both sides for $1 \leq m \leq n-1$,
so
$$
w_{m-k+1}=-(a_{k-1} w_{m-k+2} + \cdots + a_{1} w_{m} +a_0 w_{m+1} )=T_{g^\perp}(w_{m+1},w_m,\ldots,w_{m-k+2}).
$$
Also, since $g(x)$ is monic of degree $n-k$, the reverse $(w_n,w_{n-1},\ldots,w_2,w_1)$ of $w$ will start with its initial $k$ terms
being $(w_{n},w_{n-1},\ldots,w_{n-k+2},w_{n-k+1})=(0,0,\ldots,0,1)$.
In other words, this reverse of $w$ is the pseudorandom sequence of length $n$
generated by $T_{g^\perp(x)}$ with seed $(0,0,\ldots,0,1)$.
\end{proof}

\begin{proposition}
\label{simply-transitive-CSP-prop}
 Let $X=\CCC \subset \FF_q^n$ be a $k$-dimensional dual Hamming code, so that $n=q^k-1$,
with generator $g(x)$, and $w$ in $\FF_q^n$ its corresponding word. Then
\begin{enumerate}
    \item[(i)] $(X,X^{\maj}(t),C)$ exhibits the CSP if and only $\gcd(n,\cdes(w))=1$.
    \item[(ii)] In the binary case, $(X,X^{\inv}(t),C)$ exhibits the CSP if and only $\gcd(n,\wt(w))=1$.
\end{enumerate}
\end{proposition}

\begin{proof}
Combine the equivalence between Proposition~\ref{primitive-dual-Hamming-equivalences-prop} (i) and (iii) with Corollary~\ref{single-free-orbit-corollary}.
\end{proof}

This leads to the main result of this section, whose part (ii) we find surprising.

\begin{theorem}
\label{primitive-polynomial-conjecture}
Fix a positive integer $k$ and prime power $q$,
and let $n:=q^k-1$.
\begin{enumerate}

\item[(i)]
Any nonzero codeword $w$ in a dual Hamming code in $ \CCC \subset \FF_q^n$ has $\cdes(w)= \frac{q-1}{2} \cdot q^{k-1}$. 

\item[(ii)]
If $q\in \{2,3\}$, then a monic degree $k$
irreducible $g^\perp(x)$ in $\FF_q[x]$ is primitive if and only if the word $w$ corresponding to 
$g(x)=\frac{x^{n-1}}{g^\perp(x)}$ under the bijection
\eqref{codeword-to-polynomial-dictionary} has $\cdes(w)=\frac{q-1}{2} \cdot q^{k-1}$.

\item[(iii)] If $q \in \{2,3\}$, then $(X,X^\maj(t),C)$ gives a CSP
for $X=\CCC$ any dual Hamming code.

\item[(iv)] If $q=2$, then $(X,X^\inv(t),C)$ gives a CSP
for $X=\CCC$ any dual Hamming code.
\end{enumerate}
\end{theorem}

\begin{proof} 
For (i), note that part (i) of Proposition~\ref{primitive-dual-Hamming-equivalences-prop} shows that all nonzero words $w$ in $\CCC$
lie in the same $C$-orbit, while part (iv) of the same proposition implies that
the $n$-periodic extension of $w$ contains
every vector in $\FF_q^k \setminus \{\bf 0\}$ exactly once as a consecutive
subsequence each period.  Consequently, every possible pair 
$(w_{i-1},w_{i})$ (with subscripts taken modulo $n$) contributing to $\cdes(w)$ has its location uniquely determined within an $n$-period once
we 
\begin{itemize}
    \item 
choose the values $w_{i-1} > w_{i}$ in $\binom{q}{2}$ ways, and then
\item complete the length $k$ subsequence preceding it as $(w_{i-k+1},\ldots,w_{i-2},w_{i-1},w_i)$ by choosing the preceding $k-2$ entries arbitrarily in $q^{k-2}$ ways; this is not ${\bf 0}$ in $\FF_q^k$ since $w_{i-1} > w_i$.
\end{itemize}
Thus $\cdes(w)=\binom{q}{2} \cdot q^{k-2}=\frac{q-1}{2} \cdot q^{k-1}$.

For (ii), (iii), the crux is that if $q \in \{2,3\}$, then $\frac{(q-1)}{2} q^{k-1}$
is a $q$-power, so $\gcd\left(\frac{(q-1)}{2} q^{k-1},n\right)=1$.  

To deduce (iii), apply Proposition~\ref{simply-transitive-CSP-prop} to assertion (i) here.

To deduce (ii), assume $q \in \{2,3\}$ and 
$\cdes(w)=\frac{(q-1)}{2}q^{k-1}$.  We know that in $\FF_q[x]/(g^\perp(x))$, the element $\bar{x}$ 
has some multiplicative order $d$ dividing $n=q^k-1$, and want to show $d=n$.
Since the LFSR $T_{g^\perp(x)}: \FF_q^k \rightarrow \FF_q^k$ also has
order $d$, the word $w$ will be $d$-periodic, consisting of
$\frac{n}{d}$ repeats of some word of length $d$.  Hence  
$\frac{n}{d}$ divides $\cdes(w)=\frac{(q-1)}{2}q^{k-1}$.  Since
$\frac{n}{d}$ also divides $n$, it divides $\gcd(\frac{(q-1)}{2} q^{k-1},n)=1$.
Hence $d=n$ as desired.

To deduce (iv), Proposition~\ref{simply-transitive-CSP-prop} applies
once we compute the Hamming weight $\wt(w)$.  As
the $n$-periodic extension of $w$ has every binary sequence in $\FF_2^k \setminus \{\bf 0\}$ occurring exactly once consecutively in a period,
this implies $\wt(w)=2^{k-1}$, and hence $\gcd(n,\wt(w))=\gcd(2^k-1,2^{k-1})=1$, as desired.
\end{proof}

\begin{example}
The assertion of Theorem~\ref{primitive-polynomial-conjecture}(ii) fails for $q=5$ at $k=3$. The cubic irreducible 
$g^\perp(x) = 1 + x + x^3$ in $\FF[x]$ is not primitive, since $\bar{x}$ has order $d=62$ in $\FF_5[x]/(g^\perp(x))$, rather than $n=5^3-1=124$. However, one can check that 
the word $w$ corresponding to
$g(x)=\frac{x^{124}-1}{g^\perp(x)}$ still has $\cdes(w)=50=\frac{5-1}{2}\cdot 5^{3-1}$.  Likewise 
the assertion fails for $q=7$ at $k=2$.  
The irreducible quadratic
$g^\perp(x) = 6 + x + x^2$ in $\FF_7[x]$ is not primitive, as $\bar{x}$ has order $16$
in $\FF_7[x]/(g^\perp(x))$,
not $n=7^2-1=48$, but one can check that the word $w$ corresponding to $g(x)$ has 
$\cdes(w)=21=\frac{7-1}{2} \cdot 7^{2-1}$.

One can also check that the assertion of Theorem~\ref{primitive-polynomial-conjecture}(iv)
fails for $q=3$ at $k=2$, such as in Example~\ref{dual-Hamming-example} with the choice of primitive polynomial $g^\perp(x)=x^2+x+2$ as the parity check for a dual Hamming code $X=\CCC$:  no matter how one orders the alphabet $\FF_3=\{0,1,2\}$
to  define $\inv$ in $X^{\inv} (t)$, the triple $(X,X^\inv(t),C)$ does not exhibit the CSP.
\end{example}

\section{Questions}
We close with some questions that we have not seriously explored.

\begin{question} \rm
    Can one characterize the dual Hamming codes $X=\CCC$
    for which $(X,X^\maj(t),C)$ or $(X,X^\inv(t),C)$ exhibits a CSP? To what extent does this depend upon the choice of primitive polynomial parity check polynomial $g^\perp(x)$ and/or the linear ordering of $\FF_q$ used to define $\maj, \inv$?
\end{question}

\begin{question} \rm
Do other cyclic codes (e.g., 
Reed-Solomon, BCH, Golay) exhibit interesting CSPs?
\end{question}

\end{document}